\documentclass{amsart}
\usepackage{graphicx} % Required for inserting images
\usepackage{amsmath,dsfont}
\usepackage{amsthm}
\usepackage{amssymb}
\usepackage{mathrsfs}
\usepackage{graphicx}
\usepackage{enumerate}
\usepackage{tikz}
\usepackage[colorlinks,
linkcolor=red,
anchorcolor=blue,
citecolor=green
]{hyperref}
\newtheorem{theorem}{Theorem} 

\newtheorem{prop}{Proposition}
\newtheorem{lemma}[theorem]{Lemma}

\title[Falconer distance problem on Riemannian manifolds]{Falconer distance problem on Riemannian manifolds}
\author{Changbiao Jian}
\address{School of Mathematical Sciences, Zhejiang University, Hangzhou 310027, PR China}
\email{changbiaojian@zju.edu.cn}
\author{Bochen Liu}
\address{Department of Mathematics \& International Center for Mathematics, Southern
University of Science and Technology, Shenzhen, 518055, PR China}
\email{Bochen.Liu1989@gmail.com}
\author{Yakun Xi}
\address{School of Mathematical Sciences, Zhejiang University, Hangzhou 310027, PR China}
\email{yakunxi@zju.edu.cn}

\begin{document}
\begin{abstract}
We prove that on a $d$-dimensional Riemannian manifold, the distance set of a Borel set $E$ has a positive Lebesgue measure if  $$\dim_{\mathcal{H}}(E)>\frac d2+\frac14+\frac{1-(-1)^d}{8d}.$$
\end{abstract}
\maketitle

\section{Introduction}\label{section1}
For a compact subset $E \subseteq \mathbb{R}^d$, where $d \geq 2$, we define its distance set $\Delta(E)$ as:
\[\Delta(E)=\{|x-y|:x,y\in E\}.\]
In 1986, Falconer introduced a famously challenging problem in geometric measure theory and harmonic analysis \cite{Fal}. Falconer was interested 
in how large the Hausdorff dimension of $E$ needs
to be to ensure that the Lebesgue measure of $\Delta(E)$ is positive.
Falconer conjectured that the Lebesgue measure of $\Delta(E)$ is positive when $\dim_{\mathcal{H}}(E) > \frac{d}{2}$ and  proved this result under a slightly weaker assumption: $\dim_{\mathcal{H}}(E) > \frac{d + 1}{2}$. Over the following decades, this conjecture has drawn significant attention. Many remarkable partial results have been established by Mattila \cite{87}, Bourgain \cite{94}, Wolff \cite{99}, Erdogan \cite{05,06}, Du and Zhang \cite{du-zhang}, Du, Guth, Ou, Wang, Wilson and Zhang \cite{du-guth}. Nonetheless, Falconer's conjecture in its full strength is still open in all dimensions.

The best result in dimension 2 is due to Guth, Iosevich, Ou, and Wang \cite{guth2020}. They proved the Lebesgue measure of $\Delta(E)$ is positive if $\dim_{\mathcal{H}}E>\frac{5}{4}$ in the plane. Before long, their method was generalized to all even dimensions by Du, Iosevich, Ou, Wang, and Zhang \cite{du2021}. They showed that the distance set of $E\subset\mathbb R^d$ will have positive Lebesgue measure if $\dim_{\mathcal{H}}(E)>\frac{d}{2}+\frac14$ for all even $d>2$. In very recent works by Du, Ou, Ren, and Zhang \cite{du2023new,du2023weighted}, they are able to improve the dimensional exponent in every dimension $d\geq 3$. 

Exploring this conjecture on Riemannian manifolds is intriguing as well. Let $(M,g)$ be a boundaryless Riemannian manifold of dimension $d\geq2$. For a set $E\subset M$, define its distance set as
\[\Delta_g(E)=\{d_g(x,y): x,y\in E\},\]
where $d_g$ is the associated Riemannian distance function. It is then a natural question: how large must the Hausdorff dimension of $E$ be to ensure that the Lebesgue measure of $\Delta_{g}(E)$ is positive?
Similar to the work of Falconer \cite{Fal}, 
the Lebesgue measure of $\Delta_g(E)$ can be shown to be positive if 
$\dim_{\mathcal{H}}(E)>\frac{d}{2}+\frac{1}{2}$, which is
a consequence of the generalized projection theorem of Peres--Schlag \cite{R00}. Their generalized projection theorem
implies a stronger result, that is, the Lebesgue measure of the pinned distance set 
\[\Delta_{g,x}(E):=\{d_g(x,y): y\in E\},\]
for some $x\in E$, is positive under the same conditions.

Several authors have studied this Riemannian version of Falconer’s conjecture. Eswarathasan--Iosevich--Taylor \cite{R11} and Iosevich--Taylor--Uriarte--Tuero \cite{R21} studied more general distance set problems via Fourier integral operators. Their results imply the same dimensional exponent, $\frac{d}{2}+\frac{1}{2}$, as Peres--Schlag for the Riemannian case. Using local smoothing estimates of Fourier integral operators, Iosevich and Liu \cite{R19} improved Peres and Schlag's results if the pins are not required to lie in the given set itself. 
Recently, Iosevich, Liu, and Xi \cite{xi2022} extended the result of Guth, Iosevich, Ou, and Wang \cite{guth2020} in $\mathbb R^2$ to general Riemannian surfaces. However, there have not been any improvements over the classical threshold $\frac{d}{2}+\frac{1}{2}$ in higher dimensional manifolds.

{The primary purpose of this paper is to study the Falconer distance problem in higher-dimensional Riemannian manifolds.
Our main result improves the best-known bounds for the Falconer problem on general higher-dimensional Riemannian manifolds. In particular, we generalize the results of Du--Iosevich--Ou--Wang-Zhang to the Riemannian setting.

\begin{theorem}\label{t1}
    Let $d>2$ be an integer. Let $(M,g)$ be a boundaryless $d$-dimensional Riemannian manifold.   Let $E\subset M$. We assume that 
\[\dim_{\mathcal{H}}(E)>
\begin{cases}
    \frac{d}{2}+\frac{1}{4}, & d\text{ is even},\\
    \frac{d}{2}+\frac{1}{4}+\frac{1}{4d}, & d\text{ is odd}.\\
\end{cases}\]
   Then there is a point $x\in E$
     so that the Lebesgue measure of its pinned distance set $\Delta_{g,x}(E)$ is positive. Consequently, the Lebesgue measure of $\Delta_{g}(E)$ is positive.
\end{theorem}

Theorem \ref{t1} are new for all $d\ge 3$. The proof of Theorem \ref{t1}  is inspired by key strategies introduced in \cite{guth2020}, \cite{du2021}, and \cite{xi2022}. In \cite{guth2020}, the authors employ Liu's approach \cite{L19} and further improve it
by removing some ``bad" wave packets at different scales. To do so, they first make use of a radial projection of Orponen \cite{orponen19} to justify the deletion of such ``bad" wave packets. Then, they prove a refined decoupling inequality, which provides a gain once the ``bad" wave packets are removed. At first, it seems difficult to generalize this result to higher dimensions since Orponen's radial projection theorem \cite{orponen19} is only effective for sets with dimensions greater than $d-1$.
Nonetheless, in the work of Du--Iosevich--Ou--Wang--Zhang \cite{du2021}, they overcome this difficulty by using orthogonal projections to project the original measure onto a suitable lower dimensional space so that Orponen's theorem can be exploited. 
In \cite{xi2022}, refined microlocal decoupling inequalities are established on Riemannian manifolds, which are inspired by the work of Beltran, Hickman, and Sogge on variable coefficient decoupling inequalities \cite{sogge20}. Moreover, Orponen's radial projection theorem has been extended to the case of Riemannian surfaces. 

Unlike Euclidean space, where orthogonal projections are readily available, the lack of a natural analog of the orthogonal projection on Riemannian manifolds poses a major challenge. This is the main obstacle that we overcome in this paper. 
We introduce a family of submanifolds associated with each element $V$ in the Grassmannian ${\bf Gr}(\lfloor d/2\rfloor,d)$ on the tangent space of a fixed point $p_0\in M$.  
We work in Fermi coordinates with respect to the submanifold $\exp_{p_0}(V)$. For each $p\in\exp_{p_0}(V)$, we consider the submanifolds $\Sigma_{V,p}$ of dimension $\lceil d/2\rceil$ which are orthogonal 
to $\exp_{p_0}(V)$. 
Notice that $\{\Sigma_{V,p}\}_{p\in\exp_{p_0}(V)}$ is a local foliation of $M$
near the point $p_0$. 
With this foliation, we can define a natural analog of the orthogonal projection on $M$ and adopt the idea in \cite{du2021}.

This paper is organized as follows. In Section \ref{section2}, we establish the framework for studying Falconer's problem on Riemannian manifolds. We perform some standard reductions and outline the two main estimates needed. Sections \ref{section3} and \ref{section4} house the proof of these two main estimates, respectively. In Section \ref{section5}, we prove Lemma \ref{l3} and complete the proof of Theorem \ref{t1}.

\textbf{Notation.} We use the notation $A\lesssim B$, if there exists an absolute constant $C$ such that $A\leq CB$. We shall employ the notation $A\sim B$ when both $A\lesssim B$ and $B\lesssim A$. Throughout this article, $\mathrm{RapDec}(R)$ refers to quantities that rapidly decay as $R\to\infty$, that is, there is a constant $C_N$ so that $\mathrm{RapDec}(R)\leq C_NR^{-N}$ for arbitrarily large $N$. $\dim_{\mathcal{H}}$ denotes the Hausdorff dimension. $B^d(x,r)$ denotes the $d$-dimensional geodesic ball  centered at $x$ with a radius of $r$. Similarly, $S^{d-1}(x,r)$ denotes the geodesic sphere. We shall omit the dimensions if it is clear from context.

\textbf{Acknowledgement.} This project is supported by the National Key Research and Development
Program of China No. 2022YFA1007200
 and NSF China Grant No. 12131011 and 12171424. The authors would like to thank Kevin Ren for helpful conversations. The authors would like to thank an anonymous referee for helpful comments and suggestions.
\section{Standard reductions and the Main estimates}\label{section2}
Let $(M,g)$ be a $d$-dimensional Riemannian manifold, and let $\alpha>0$. Consider a compact set $E\subset M$ with Hausdorff dimension $>\alpha$. By possibly rescaling the metric, we assume that $E$ is contained in the unit geodesic ball centered at a point $p_0$, so that in the geodesic normal coordinate about $p_0$, the metric $g$ is approximately flat, in the sense that for all $i,j$, $|g_{ij}-\delta_{ij}|<\epsilon_0$ for some universal constant $\epsilon_0$ that we can choose to be small. We select two disjoint compact subsets $E_1\subset B_1$ and $E_2\subset B_2$ of $E$ such that
\[\dim_{\mathcal{H}}(E_1), \dim_{\mathcal{H}}(E_2)>\alpha\]
and  $d_g(E_1,E_2)\sim1$. Here, $B_1$ and $B_2$ denote two geodesic balls with radius $\frac{1}{100}$, and the distance between $E_1$ and $E_2$ is defined as
\[d_g(E_1,E_2):=\inf_{x\in E_1,y\in E_2}d_g(x,y)>0.\]
Additionally, each $E_j$ admits a probability measure $\mu_j$ such that $\mathrm{supp}~\mu_j\subset E_j$ and $\mu_j(B(x,r))\lesssim r^{\alpha}$. 

For any fixed $x\in E_2$, we define the pinned distance function associated with $x$ as 
\[d^x_g(y) := d_g(x,y).\]
We then define the pushforward measure $d^x_*(\mu) = (d_g)^x_*(\mu)$ of $\mu$ through the relation:
\[\int_{\mathbb{R}}h(t)\,d^x_*(\mu)=\int_Mh(d_g(x,y))\,d\mu(y).\]

Next, we establish the framework of our problem for the Riemannian setting and construct a complex-valued measure $\mu_{1, \text{good}}$, which corresponds to the ``good" part of $\mu_1$ concerning $\mu_2$. Then, we shall investigate the pushforward of this measure under  $d^x_g(y)$ and reduce the problem to two main estimates.

Let us now proceed to set up the problem. We start with a microlocal decomposition,  a higher dimensional analog of that in \cite{xi2022}.  We work in the geodesic normal coordinates ${(x', x_d)}\in\mathbb{R}^{d-1}\times\mathbb{R}$ about the point $p_0$. We may assume that $p_0$ lays near the middle point between $B_1$ and $B_2$ so that $E_1\subset B_1$ and $E_2\subset B_2$ are approximately $1$-separated by their $d$-th coordinate.

Let $R_0$ be a sufficiently large number, denoted by  $R_j=2^jR_0$. Cover the annulus region $R_{j-1}\leq|\xi|\leq R_j$ where $\left\{(\xi',\xi_d): \frac{|\xi_d|}{|\xi|}\geq\frac{1}{10}\right\}$ by rectangular blocks  $\tau$. These blocks have dimensions $R_j^{1/2}\times\cdots\times R_j^{1/2}\times R_j$, where the long direction of each block aligns with the radial direction of the annulus. To account for the remaining portion of the annulus, we use two larger blocks to cover them. 
The corresponding smooth partition of unity subordinate to this cover is the following.
\[1=\psi_0+\sum_{j\geq1}\big[\sum_{\tau}\psi_{j,\tau}+\psi_{j,-}+\psi_{j,+}\big],\]
 where $\psi_{j,\tau}$ is supported in $\tau$, $\psi_{j,+}$ is a smooth bump function adapted to the set:
 \[\Big\{(\xi',\xi_d): R_{j-1}\leq|\xi|\leq R_j, 0\leq\frac{\xi_d}{|\xi|}\leq\frac{1}{10}\Big\}.\]
 Similarly $\psi_{j,-}$ is a smooth bump function  adapted to the set:
 \[\Big\{(\xi',\xi_d): R_{j-1}\leq|\xi|\leq R_j, -\frac{1}{10}\leq\frac{\xi_d}{|\xi|}\leq0\Big\}.\]

Let $\delta>0$ be a small constant. Denote $T_0$ by the geodesic tube around the $x_d$-axis with cross-section radius  $\frac{1}{10}$. For each pair $(j,\tau)$, we consider the collection of geodesics $\gamma$ such that each $\gamma$ intersects the hyperplane $\{x_d=0\}$ and the tangent vector at the intersection point is pointing in the direction of $\tau$ in our coordinate system. For each such $\gamma$, we define the associated tube $T$ to be a geodesic tube with central geodesic $\gamma$ of dimension $R_j^{-1/2+\delta}\times\cdots\times R_j^{-1/2+\delta}\times1$. We then choose a subcollection of such tubes, $\mathbb{T}_{j,\tau}$, so that tubes in $\mathbb{T}_{j,\tau}$  cover $T_0$ with bounded overlap. We then define $\mathbb{T}_j=\cup_{\tau}\mathbb{T}_{j,\tau}$ as well as $\mathbb{T}=\cup_{j}\mathbb{T}_j$. Let $\eta_T$ be a smooth partition of unity subordinate to this covering. That is, for each pair $(j, \tau)$, we have $\sum_{T\in\mathbb{T}_{j,\tau}}\eta_T$ equals to 1 in the unit ball.

Similar to \cite{xi2022}, we need a microlocal decomposition for functions with support in $B_1$. Fix $(j,\tau)$, for each $(x,\xi)$, there exists a unique point $z=(z'(x,\xi),0)$ on the hyperplane $\{x_d=0\}$ such that the geodesic connecting $z$ and $x$ has tangent vector $\xi/|\xi|$ at $x$. We then define $\Phi(x,\xi/|\xi|)$ to be the unit tangent vector of this geodesic at the point $z$ and extend it to be a homogeneous function of degree 1 in $\xi$ so that $|\Phi(x,\xi)|=|\xi|$.  Let $f$ be a function with $\mathrm{supp}(f)\subset B_1$, and for each $(j,\tau)$, we define
\[\mathcal P_{j,\tau}f(x)=\iint_{\mathbb{R}^d\times\mathbb{R}^d}e^{2\pi i(x-y)\cdot\xi}\psi_{j,\tau}(\Phi(x,\xi))f(y)\,dyd\xi.\]

Similarly, we define
\[\mathcal P_{j,\tau}f(x)=\iint_{\mathbb{R}^d\times\mathbb{R}^d}e^{2\pi i(x-y)\cdot\xi}\psi_{j,\pm}(\Phi(x,\xi))f(y)\,dyd\xi,\]
and
\[\mathcal P_{j,0}f(x)=\iint_{\mathbb{R}^d\times\mathbb{R}^d}e^{2\pi i(x-y)\cdot\xi}\psi_{j,0}(\Phi(x,\xi))f(y)\,dyd\xi.\]
Consequently, $\mathcal P_0+\sum_{j\geq1}[\sum_{\tau}\mathcal P_{j,\tau}+\mathcal P_{j,+}+\mathcal P_{j,-}]$ is the identity operator.

Now, for each $T\in\mathbb{T}_{j,\tau}$, we define 
\[M_Tf:=\eta_T\mathcal P_{j, \tau}f.\]
Similarly, we define $M_{j,\pm}f=\eta_{T_0}\mathcal P_{j,\pm}f$ and $M_0f=\eta_{T_0}\mathcal P_0f$, where $\eta_{T_0}$ is a smooth function satisfying $\eta_{T_0}=1$ on $T_0$ and  decays rapidly away from it.

With these preparations in place, we now construct the ``good" measure $\mu_{1, good}$ by removing certain ``bad" wave packets from $\mu_1$.  We shall show that the error between $d^x_*(\mu_{1,\text{\rm good}})$ and $d^x_*(\mu_1)$ remains sufficiently  small in the sense of $L^1$ norm.
Denote $4T$ as a concentric tube of $T$ with four times the cross-section radius. Consider a large constant $c(\alpha)=c(\alpha,d)>0$ which shall be determined later. { Now we define the bad tubes for the odd and even dimensional cases, respectively. When $d$ is even, for each tube $T\in\mathbb{T}_{j,\tau}$, we say that it is a bad tube if
\[\mu_2(4T)\geq R_j^{-d/4+c(\alpha)\delta}.\]
When $d$ is odd, we say that $T$ is a bad tube if
\[\mu_2(4T)\geq R_j^{-(d-1)/4+c(\alpha)\delta}.\]}
Otherwise, we consider $T$ as a good tube. Accordingly, we define
\[\mu_{1,\text{\rm good}}:=M_0\mu_1+\sum_{T\in\mathbb{T},T~\text{is good}}M_T\mu_1.\]
 
In \cite{guth2020}\cite{du2021}, the authors established that good tubes make primary contributions in even-dimensional Euclidean space. We generalize this result to the Riemannian setting. We will prove the following two propositions in the Riemannian framework. Indeed, Theorem \ref{t1} is a consequence of these two propositions.
\begin{prop}\label{p1}
Let $d\geq 2$ be an integer and $(M,g)$ be a boundaryless Riemannian 
    manifold of dimension $d$. If  $\alpha>\frac{d}{2}$ and $R_0$ is large enough, then there exists a subset $E'_2\subset E_2$ such that $\mu_2(E'_2)\geq1-\frac{1}{1000}$ and for any $x\in E'_2$, 
\[\|d^x_*(\mu_1)-d^x_*(\mu_{1,\text{\rm good}})\|_{L^1}<\frac{1}{1000}.\]
\end{prop}
\begin{prop}\label{p2}
    Let $d\geq 2$ be an  integer. { Suppose that\[\dim_{\mathcal{H}}(E)>
\begin{cases}
    \frac{d}{2}+\frac{1}{4}, & d\text{ is even},\\
    \frac{d}{2}+\frac{1}{4}+\frac{1}{4d}, & d\text{ is odd}.\\
\end{cases}\]} Then for sufficiently small $\delta=\delta(\alpha)$, we have
    \[\int_{E_2}\|d^x_*(\mu_{1,\text{\rm good}})\|_{L^2}^2\,d\mu_2(x)<+\infty.\]
\end{prop}

We now apply the above two propositions to establish the main theorem. We shall give the proof of Proposition \ref{p1} and Proposition \ref{p2} in Section \ref{section3} and Section \ref{section4} respectively.
\begin{proof}[Proof of Theorem \ref{t1}] The proof is identical to that in \cite{guth2020,xi2022,du2021}. We include it here for completeness.
   \begin{align*}
       \int_{\Delta_{g,x}(E)}|d^x_*(\mu_{1,\text{\rm good}})|&=\int|d^x_*(\mu_{1,\text{\rm good}})|-\int_{\Delta_{g,x}(E)^c}|d^x_*(\mu_{1,\text{\rm good}})|\\
      &\geq\int|d^x_*(\mu_{1})|-2\int|d^x_*(\mu_1)-d^x_*(\mu_{1,\text{\rm good}})|\\
      &\geq1-\frac{2}{1000}.
   \end{align*} 
Here, we have used the fact that $d^x_*(\mu_1)$ is a probability measure. The final inequality is a direct consequence of Proposition \ref{p1}.

Conversely, employing Proposition \ref{p2} and Cauchy-Schwarz, we obtain
   \[\Big(\int_{\Delta_{g,x}(E)}|d^x_*(\mu_{1,\text{\rm good}})|\Big)^2\leq |\Delta_{g,x}(E)|\cdot\|d^x_*(\mu_{1,\text{\rm good}})\|^2_{L^2}\lesssim|\Delta_{g,x}(E)|.\]
  Hence, we conclude that $|\Delta_{g,x}(E)|$ is positive.
\end{proof}

\section{ proof of Proposition \ref{p1}}\label{section3}
In this section, we establish Proposition \ref{p1} by studying the pushforward measures $d^x_*(\mu_1)$ and $d^x_*(\mu_{1,\text{\rm good}})$. We will apply the geometric properties of bad tubes to control $\|d^x_*(\mu_1)-d^x_*(\mu_{1,\text{\rm good}})\|_{L^1}$. To commence, let us recall the following lemma, which is essentially from \cite{xi2022}.
\begin{lemma}\label{l2}
For each point $x\in E_2,$  define
\begin{equation*}\text{\rm Bad}_j(x):=\mathop{\bigcup}_{T\in\mathbb{T}_j:\, x\in 2T~ \text{and}~T~\text{is}~ \text{bad}}2T,~~~\forall j\geq1.
\end{equation*}
    Then one has
    \[\|d^x_*(\mu_{1,\text{\rm good}})-d^x_*(\mu_1)\|_{L^1}\lesssim\sum_{j\geq1}R^{\delta}_j\mu_1(\text{\rm Bad}_j(x))+\mathrm{RapDec}(R_0).\]
\end{lemma}
\begin{proof}
The only differences between the above lemma and \cite[Lemma 7.3]{xi2022} are in the definitions of bad tubes and the dimensions of the underlying manifolds.
  By the arguments in \cite{xi2022} (and \cite{guth2020,du2021}), it suffices to prove that for $x\in E_2$ and $x\notin 2T$,
   \begin{equation}\label{l21}
   \|d^x_*(M_T\mu_1)\|_{L^1}\lesssim \mathrm{RapDec}(R_j),
   \end{equation}
   and for any $x\in E_2$,
  \begin{equation}\label{l23}
  \|d^x_*(M_{j,\pm}\mu_1)\|_{L^1}\lesssim \mathrm{RapDec}(R_j).
  \end{equation}
 
Now we prove \eqref{l21}.
Recall that
  \[d^x_*(M_T\mu_1)(t)=\int_{S^{d-1}(x,t)}M_T\mu_1(y)\,d\sigma(y),\]
where $\sigma(y)$ is the formal restriction of $\mu_2$ on the geodesic sphere \[S^{d-1}(x,t):=\{y: d_g(x,y)=t\}.\] Noting the separation of $E_1$ and $E_2$, we may assume that
$t\approx1$, otherwise  $d^x_*(M_T\mu_1)(t)$ is negligible. Denote
\[A_{j,\tau}(x,\xi)=\psi_{j,\tau}(\Phi(x,\xi)).\]
Since $|\hat{\mu}_1(\xi)|\leq1$ and
\begin{align*}M_T\mu_1&=\eta_T(y)\iint e^{2\pi i(y-z)\cdot\xi}A_{j,\tau}(z,\xi)\,d\mu_1(z)d\xi\\
&=\eta_T(y)\int e^{2\pi iy\xi}\hat{\mu}_1(\xi)A_{i,\tau}(y,\xi)\,d\xi,
\end{align*}
It suffices to show that  one has, uniformly in $\xi$,
\begin{equation}\label{l22}
\bigg|\int_{S^{d-1}(x,t)}\eta_T(y)A_{j,\tau}(y,\xi)e^{2\pi iy\cdot\xi}\,d\sigma(y)\bigg|\lesssim \mathrm{RapDec}(R_j).
\end{equation}
The tube $T$ intersects $S^{d-1}(x,t)$ in at most two caps, which can be dealt with in similar fashions. It suffices to consider a single cap. 
Denote by $y_0$ the center of the cap.
Since $x\notin 2T,\ y\in T\cap S^{d-1}(x,t)$ and $\xi$ is essentially in $\tau$,
we have that the angle between vectors $\frac{y-x}{|y-x|}$ and $\frac{\xi}{|\xi|}$
is greater than $R_j^{-1/2+\delta}.$ By a change of variable $z=y-x$, we see that \eqref{l22}
equals 
\begin{equation}\label{l33}
\bigg|e^{2\pi ix\cdot\xi}\int_{S^{d-1}(0,t)}\eta_T(z+x)A_{j,\tau}(z+x,\xi)e^{2\pi iz\cdot\xi}\,d\sigma(z)\bigg|.
\end{equation}
Without loss of generalities, in geodesic normal coordinates about $x$, we may assume that $z_0=y_0-x$ is given by
\[z_0=y_0-x=z(0,\theta_0)=(0,\ldots, 0,t\cos\theta_0,t\sin\theta_0),\]
 and
\[\frac{\xi}{|\xi|}=(0,\ldots, 0,\cos\omega,\sin\omega),\]
for some $\theta_0,\omega\in[0,2\pi).$

More generally, we employ the following change of coordinates:
\begin{align*}B^{d-2}(0,t)\times[0,2\pi)&\to S^{d-1}(0,t):\\(m,\theta)&\to z(m,\theta):=\big(m,\sqrt{t^2-|m|^2}\cos\theta,\sqrt{t^2-|m|^2}\sin\theta\big).\end{align*}
The associated Jacobian is $\approx 1$ in our domain of integration since $t\approx1$, thus \eqref{l33} is reduced to
\begin{equation}\label{l44}
    \bigg|\int_{B^{d-2}(0,t)}\int^{2\pi}_0\eta_T(z(m,\theta)+x)A_{j,\tau}(z(m,\theta)+x, \xi)e^{2\pi i|\xi|\sqrt{t^2-|m|^2}\cos(\omega-\theta)}\,d\theta dm\bigg|.
\end{equation}
Since the intersection angle between $\frac{y-x}{|y-x|}$ and $\frac{\xi}{|\xi|}$
is greater than $R_j^{-1/2+\delta}$, 
we have $|\theta_0-\omega|\gtrsim R_j^{-1/2+\delta}$. Furthermore, 
we have 
$|\theta-\omega|\gtrsim R_j^{-1/2+\delta}$ if $z\in S^{d-1}(0,t)$ and $z+x\in T$.
Since the cap of $T\cap S^{d-1}(x,t)$ has  dimensions $ R^{-1/2+\delta}\times\cdots\times R^{-1/2+\delta}_j$ and center $z_0=(0,t\cos\theta_0,t\sin\theta_0)$, we have $|m|\lesssim R_j^{-1/2+\delta}$. Note that $t\approx1$, so we get 
\[\frac{d}{d\theta}(|\xi|\sqrt{t^2-m^2}\cos(\omega-\theta))=|\xi|\sqrt{t^2-m^2}\sin(\theta-\omega)\gtrsim R_j^{1/2+\delta}.\]
Since \[\partial^{(k)}_{\theta}[\eta_T(z(m,\theta)+x)A_{j,\tau}(z(m,\theta)+x,\xi)]\lesssim R_j^{k(1/2-\delta)(d-1)},\]
an integration by parts argument in the $\theta$ variable shows that \eqref{l44} decays rapidly in $R_j$.

The bound \eqref{l23} can be proved similarly if we replace $T$ by $T_0$.
\end{proof}

Applying lemma \ref{l2}, Proposition \ref{p1} is reduced to estimating the measure of $\text{\rm Bad}_j(x)$.
We define \[\text{\rm Bad}_j:=\{(x_1,x_2)\in E_1\times E_2: \text{there is a bad tube } T\in\mathbb{T}_j~\text{so that }~2T~\text{contains}~x_1~\text{and}~x_2\},\]
The following lemma holds.
\begin{lemma}\label{l3}
    Let $d\geq 2$ be an  integer. If $\alpha>\frac{d}{2}$, then for sufficiently large $c(\alpha)>0$ and  every $j\geq1,$  we have
    \[\mu_1\times\mu_2(\text{\rm Bad}_j)\lesssim R_j^{-{2\delta}}.\]
\end{lemma}
We shall postpone the proof of Lemma \ref{l3} to Section \ref{5} and show that Proposition \ref{p1} follows from  Lemma \ref{l3}. 
\begin{proof}[Proof of Proposition \ref{p1}]
Notice that
\[\mu_1\times\mu_2(\text{\rm Bad}_j)=\int\mu_2(\text{\rm Bad}_j(x))\,d\mu_1(x).\]
For every $j\geq1$, one can choose $A_j\subset E_2$ such that 
$\mu_2(A_j)\leq R_j^{-(1/2)\delta}$ and 
\[\mu_1(\text{\rm Bad}_j(x))\lesssim R_j^{-(3/2)\delta},\qquad \forall x\in E_2\backslash A_j.\] 
We have \[\mu_2(E_2\backslash\cup_{j\geq1}A_j)\geq1-\frac{1}{1000},\] provided that $R_0$ is large enough. Thus for $x\in E'_2$, we get
\[\|d^x_*(\mu_1)-d^x_*(\mu_{1,\text{\rm good}})\|_{L^1}\lesssim\sum_{j\geq1}R^{\delta}_j\mu_1(\text{\rm Bad}_j(x))+{\rm RapDec}(R_0)\lesssim R_0^{-\delta/2}\leq \frac{1}{1000}.\]
\end{proof}

\section{Proof of Proposition \ref{p2}}\label{section4}
To establish Proposition \ref{p2}, we need the refined microlocal decoupling inequality proven in \cite{xi2022}. We shall follow the definitions and notation in \cite[Section 2]{xi2022}.

Given $1\le R\le\lambda$ and $\delta>0$. { Let $\phi(x,y)$ be a phase function supported in the unit ball satisfying: 1. the Carleson-Sj\"olin condition; 2. all principal curvatures of the $C^{\infty}$-hyperfaces
$$S_{x_0}=\{\nabla_x\phi(x_0,y): a(x_0,y)\ne0\},\qquad
S_{y_0}=\{\nabla_y\phi(x,y_0): a(x,y_0)\ne0\}$$
are nonzero and share the same sign. Note that this convexity condition is not needed in the surface case since there is only one principal curvature.
We denote by $\phi^\lambda({}\cdot{},{}\cdot{}):=\phi({}\cdot/\lambda{},{}\cdot/\lambda{})$ the rescaling of $\phi$. Similarly, $a^\lambda({}\cdot{},{}\cdot{}):=a({}\cdot/\lambda{},{}\cdot/\lambda{})$ denotes the rescaling of the corresponding amplitude.}
Let $N_{\phi^{\lambda},R}$ and $N^{\tau}_{\phi^{\lambda},R}$ be as defined in \cite[Equation (2.5)]{xi2022}. A function $f$ is said to be microlocalized to a tube $T$ if it has microlocal support essentially in $N^{\tau}_{\phi^{\lambda}, R}$ and is essentially supported in $T$ in physical space as well.
We partition $\mathbb{S}^{d-1}$ into $R^{-1/2}$-caps denoted by $\tau$. Now, we define a collection of curved tubes $\mathbb{T}=\mathbb{T}_{\tau}$ associated with the phase function $\phi^{\lambda}$ and the caps $\tau$. A curved tube $T$ with a length of $R$ and a cross-sectional radius of $R^{1/2+\delta}$ is in $\mathbb{T}_{\tau}$ if its central axis is the curve $\gamma^{y_0}_{\tau}$, defined by
\[\gamma^{y_0}_{\tau}=\left\{x:-\frac{\nabla_y\phi^{\lambda}(x,(y_0,0))}{|\nabla_y\phi^{\lambda}(x,(y_0,0))|}=\theta\right\},\]
where $y_0$ is chosen from a maximal $R^{1/2+\delta}$ separated subset of $\mathbb{R}^{d-1}$ and $\theta$ is the center of $\tau$. This microlocal decomposition is consistent with our wavepacket decomposition in Section \ref{section2}.
Indeed, in the Riemannian setting, $\phi=d_g$, and $\theta$ is the tangent vector of $\gamma^{y_0}_{\tau}$ at $(y_0,0)$. We recall the refined microlocal decoupling theorem established in \cite{xi2022} below.
\begin{theorem}[Theorem 5.1 in \cite{xi2022}]\label{t5}
  Let $2\leq p\leq\frac{2(d+1)}{d-1}$ and $1\leq R\leq\lambda$. Let $f$ be a function with microlocal support in $N_{\phi^{\lambda}, R}$. Let $\mathbb{T}$ denote a collection of tubes associated with $R^{-1/2}$-caps denoted by $\tau$, and $\mathbb{W}\subset\mathbb{T}$. For every $T\in\mathbb{W}$, we assume that $T$ is contained in $B_{R}$.  Denote $W$ as the cardinality of $\mathbb{W}$. Additionally, we assume that  $f=\sum_{T\in\mathbb{W}}f_T$, where each $f_T$ is microlocalized to $T$, and $\|f_T\|_{L^p}$ is roughly constant among all  $T\in\mathbb{W}$.
Let $Y$ be a collection of $R^{1/2}$-cubes in $B_R$, each of which intersects at most $L$ tubes $T\in\mathbb{W}$.
    Then
    \begin{equation}
        \|f\|_{L^p(Y)}\lesssim_{\varepsilon}R^{\varepsilon}\Big(\frac{L}{W}\Big)^{1/2-1/p}\Big(\sum_{T\in\mathbb{W}}\|f_T\|_{L^p}^2\Big)^{1/2}.
    \end{equation}
\end{theorem}

Now, we are ready to prove Proposition \ref{p2}.
The proof is very similar to that in \cite{xi2022}. For completeness and convenience
of readers, we include the argument here.

\begin{proof}[Proof of Proposition \ref{p2}]
{ We first assume that $d$ is even.}
Recalling the definition of $d^x_*(\mu_{1,\text{\rm good}})$, we have
\[d^x_*(\mu_{1,\text{\rm good}})(t)\approx\int_{d_g(x,y)=t}\mu_{1,\text{\rm good}}(y)\,dy=\iint e^{-2\pi i(d_g(x,y)-t)\tau}\,d\tau\mu_{1,\text{\rm good}}(y)\,dy.\]
By Plancherel's theorem,  we get
\[\|d^x_*(\mu_{1,\text{\rm good}})\|_{L^2}^2\approx\int\Big|\int e^{-2\pi i\lambda d_g(x,y)}\mu_{1,\text{\rm good}}(y)\,dy\Big|^2d\lambda.\]
Recall that $R_j=2^jR_0$. It suffices to prove that for every $j$, 
\begin{equation}\label{gg5}
\Big|\iint_{\lambda\approx R_j}\Big|\int e^{-2\pi i\lambda d_g(x,y)}\mu_{1,\text{\rm good}}(y)\,dy\Big|^2\,d\lambda\,d\mu_2(x)\Big|\lesssim 2^{-j\varepsilon} 
\end{equation}
holds for some $\varepsilon>0.$
Noting that  $\lambda\approx R_j,$  by a standard integration by parts argument, we reduce \eqref{gg5} to estimating
\begin{equation}\label{b2}
\iint_{\lambda\approx R_j}\Big|\int e^{-2\pi i\lambda d_g(x,y)}\mu^j_{1,\text{\rm good}}(y)\,dy\Big|^2\,d\lambda\,d\mu_2(x),
\end{equation}
where 
\[\mu^j_{1,\text{\rm good}}:=\sum_{T\in\mathbb{T}_j, T ~\text{is good}}M_T\mu_1.\]
Denote
\[F(x):=\int e^{-2\pi i\lambda d_g(x,y)}\mu^j_{1,\text{\rm good}}(y)\,dy,\]
and
\[F_T(x):=\int e^{-2\pi i\lambda d_g(x,y)}M_T\mu_1(y)\,dy.\]
It follows that
\[F(x)=\sum_{T\in\mathbb{T}_j, T~\text{is good}}F_T(x).\]
Now we choose $R=\lambda^{1-2\delta}\approx R_j^{1-2\delta}.$  One sees that each $T$ is a geodesic tube of dimensions $R^{-1/2}\times\cdots\times R^{-1/2}\times 1$. After rescaling, it is straightforward to check that both $F$ and $F_T$ satisfy the assumptions of Theorem \ref{t5}. We refer the readers to \cite{xi2022} for more detail. 

Let $p_c:=\frac{2(d+1)}{d-1}$. By dyadic pigeonholing, it suffices to consider tubes $T\in \mathbb{W}$ where $\|F_T\|_{L^{p_c}}\sim\beta>0.$  Let $W$ denote as the cardinality of $\mathbb{W}$.

Now we do the integration over $d\mu_2(x)$. To this end, we introduce a non-negative radial bump function $\rho_{j}\in C^{\infty}_0$ such that it equals to $1$ in $B_{10R_j}(0)$ and $0$ outside $B_{20R_j}(0)$. Notice that $F$ is supported in the unit ball  with $\widehat{F}(\xi)(1-\rho_{j})(\xi)=\mathrm{RapDec}(|\xi|)$. Therefore
\[F=(\widehat{F}\cdot\rho_{j})^{\vee}+\mathrm{RapDec}(R_j)=F*\widehat{\rho_{j}}+\mathrm{RapDec}(R_j).\]
We have
\begin{align*}
    \int|F(x)|^2\,d\mu_2(x)&=\int|F*\widehat{\rho_{j}}(x)|^2\,d\mu_2(x)+\mathrm{RapDec}(R_j)\\
    &\lesssim\int |F|^2*\widehat{\rho_{j}}(x)\,d\mu_2(x)+\mathrm{RapDec}(R_j)\\
    &=\int_{B_1(0)}|F(x)|^2|\widehat{\rho_{j}}|*\mu_2(x)\,dx+\mathrm{RapDec}(R_j).
\end{align*}
Let $\mu_{2,j}:=|\widehat{\rho_{j}}|*\mu_2$. Divide $B_1(0)$ into $R^{-1/2}$-cubes $Q$. Using 
dyadic pigeonholing again, it suffices to consider
\[\mathcal{Q}_{r,L}=\Big\{Q:\mu_{2,j}(Q)\sim r~\text{and}~Q~\text{intersects}~\sim L~\text{tubes}~T\in\mathbb{W}\Big\},\]
and further restrict the integration domain of $|F|^2$ from $B_1(0)$ to $Y_{r,L}=\cup_{Q\in\mathcal{Q}_{r,L}}Q.$

Applying H\"older's inequality, we get
\begin{align}\label{b5}
    \int_{Y_{r,L}}|F|^2\mu_{2,j}(x)\,dx\lesssim\bigg(\int_{Y_{r,L}}|F(x)|^{p_c}\bigg)^{2/{p_c}}
    \bigg(\int_{Y_{r,L}}\mu_{2,j}(x)^{{p_c}/({p_c}-2)}\,dx\bigg)^{({p_c}-2)/{p_c}}.
\end{align}
By Theorem \ref{t5}, one can bound the first factor by
\begin{equation}\label{b3}
    \big\|F\big\|_{L^{p_c}(Y_{r,L})}\lesssim R^{\varepsilon}\left(\frac{L}{W}\right)^{\frac{1}{2}-\frac{1}{{p_c}}}\bigg(\sum_{T\in\mathbb{W}_{\beta}}\|F_T\|_{L^{p_c}}^2\bigg)^{\frac{1}{2}}.
\end{equation}
To estimate the second factor, we note that 
$|\widehat{\rho_{j}}|$ decays rapidly outside $B_{R_j^{-1}}$, and thus
\begin{equation*}
    \mu_{2,j}(x)=\int|\widehat{\psi}_{R_j}(x-y)|\,d\mu_2(y)\lesssim R^d_j\cdot\mu_2(B_{R_j^{-1}})+\mathrm{RapDec}(R_j)\lesssim R^{d-\alpha}_j.
\end{equation*}
So we have
\begin{equation}\label{b4}
    \int_{Y_{r,L}}\mu_{2,j}(x)^{{p_c}/({p_c}-2)}\,dx\lesssim R_j^{2(d-\alpha)/({p_c}-2)}\mu_{2,j}(Y_{r,L}).
\end{equation}
Combining \eqref{b5}, \eqref{b3} and \eqref{b4}, by Lemma 8.1 in \cite{xi2022} and the definition of the bad tube, we
have
\begin{equation}\label{b6}
\begin{aligned}
    \int_{\lambda\approx R_j}\int|F(x)|^2\,d\mu_2(x)\,d\lambda&\lesssim_{\delta} R_j^{O(\delta)+(\frac{5}{2{p_c}}-\frac{1}{4})d-\frac{2\alpha}{{p_c}}}\sum_{T\in \mathbb{T}_j}\int_{\lambda\approx R_j}\|F_T\|^2_{L^{p_c}}\,d\lambda\\
    &\lesssim_{\delta} R_j^{O(\delta)+(\frac{5}{2{p_c}}-\frac{1}{4})d-\frac{2\alpha}{{p_c}}}\cdot |T|^{2/{p_c}}\sum_{T\in \mathbb{T}_j}\int_{\lambda\approx R_j}\|F_T\|^2_{L^{\infty}}\,d\lambda.
\end{aligned}
\end{equation}
Here, in the last inequality, we have used the fact that $F_T$ is essentially supported on
$2T$.

To bound the right-hand side of \eqref{b6}, we recall that 
$$F_T(x)=\int\Big(\int_{2T}e^{-2\pi i(\lambda d_g(x,y)y\cdot\xi)}\,dy\Big)\widehat{M_T
\mu_1}(\xi)\,d\xi.$$
This means that $\xi $ essentially lies in the $\lambda/R\approx R_j^{2\delta}$-neighborhood of $\lambda\tau,$ where $\tau \subset\mathbb{S}^{d-1}_y$ is a $R^{-1/2}$-cap. By Cauchy--Schwartz, we get
\begin{equation}\label{ee10}
\begin{aligned}
\|F_T\|^2_{L^{\infty}}&\lesssim\bigg(\int_{2T}\int|\widehat{M_T\mu_1}(\xi)|\psi^{\tau}_{\phi^{\lambda},R}(\lambda y,\xi /\lambda)\,d\xi dy\bigg)^2\\
&\lesssim\int_{2T}\int|\widehat{M_T\mu_1}(\xi)|^2\psi^{\tau}_{\phi^{\lambda},R}(\lambda y,\xi/\lambda)\,d\xi dy\int_{2T}\int\psi^{\tau}_{\phi^{\lambda},R}(\lambda y,\xi/y)\,d\xi dy.
\end{aligned}
\end{equation}
Note that for fixed $y\in 2T$,
\[\int\psi^{\tau}_{\phi^{\lambda},R}(\lambda y,\xi/y)\,d\xi\lesssim R^{O(\delta)}_j\cdot\lambda^{d-1}\cdot R^{-(d-1)/2},\]
So we have
\begin{equation}\label{ee11}
\int_{2T}\int\psi^{\tau}_{\phi^{\lambda},R}(\lambda y,\xi/y)\,d\xi dy\lesssim R^{O(\delta)}_j\cdot\lambda^{d-1}\cdot R^{-(d-1)/2}\cdot |T|\lesssim R_j^{O(\delta)}. 
\end{equation}
By the definition of $\psi^{\tau}_{\phi^{\lambda},R}$, we, meanwhile,  have the following
\begin{equation}\label{ee12}
\int_{\lambda\approx R_j}\psi^{\tau}_{\phi^{\lambda},R}(\lambda y,\xi/\lambda)\,d\lambda\lesssim R_j^{O(\delta)},
\end{equation}
uniformly holds in $y\in 2T$.
Combine \eqref{ee10},\eqref{ee11} and \eqref{ee12} and recall that $\widehat{M_T\mu_1}(\xi)$ is independent of $\lambda$, it follows that
\[\int_{\lambda=R_j}\|F_T\|^2_{L^{\infty}}\,d\lambda\lesssim R_j^{O(\delta)}\cdot|T|\cdot\int|\widehat{M_T\mu_1}(\xi)|^2\,d\xi.\]
Recall that $M_T=\eta_T\mathcal P_{j,\tau}$. Since $\mathcal P_{j,\tau}$ behaves like a $0$-th order pseudodifferential operator, $M_T$ is bounded on $L^2$.
By invoking Plancherel twice, we have
\begin{align*}  
\sum_{T\in\mathbb{T}_j}\int_{\lambda\approx R_j}\|F_T\|^2_{L^{\infty}}\,d\lambda
&\lesssim R_j^{O(\delta)}\cdot R^{-(d-1)/2}\int_{|\xi|\approx R_j}|\hat{\mu}_1(\xi)|^2\,d\xi,
\end{align*}
Inserting this bound into \eqref{b6} and recalling that $p_c=\frac{2(d+1)}{d-1}$, we have
\begin{align*}  \int_{\lambda\approx R_j}\int|F(x)|^2\,d\mu_2(x)\,d\lambda&\lesssim
R_j^{O(\delta)-\frac{d}{2(d+1)}-\frac{(d-1)\alpha}{d+1}}\int_{|\xi|\approx R_j}|\hat{\mu}_1(\xi)|^2\,d\xi\\
&\lesssim R_j^{O(\delta)+d-\alpha-\frac{d}{2(d+1)}-\frac{(d-1)\alpha}{d+1}}I_{\alpha-\delta}(\mu_1).
\end{align*}
Here
\[I_{s}(\mu_1)=\iint|x-y|^{-s}\,d\mu_1(x)\,d\mu_2(y)=c_{s,d}\int_{\mathbb{R}^d}
|\xi|^{-(d-s)}|\hat{\mu}_1(\xi)|^2\,d\xi,\]
denotes the energy integral of $\mu_1$, which is finite for any $s\in(0,\alpha)$ (see, e.g., \cite{mattila2015}).
We conclude the proof by noting that $d-\alpha-\frac{d}{2(d+1)}-\frac{(d-1)\alpha}{d+1}<0$ if  $\alpha$ is greater than $\frac{d}{2}+\frac{1}{4}$ and $\delta$ is sufficiently small.

{ If $d$ is odd, recall that $T$ is a good tube if
\[\mu_2(4T)\leq R_j^{-(d-1)/4+c(\alpha)\delta},\]
then by the same discussion as above, we have the following estimates.
\begin{align*}  \int_{\lambda\approx R_j}\int|F(x)|^2\,d\mu_2(x)\,d\lambda&\lesssim
R_j^{O(\delta)-\frac{d-1}{2(d+1)}-\frac{(d-1)\alpha}{d+1}}\int_{|\xi|\approx R_j}|\hat{\mu}_1(\xi)|^2\,d\xi\\
&\lesssim R_j^{O(\delta)+d-\alpha-\frac{d-1}{2(d+1)}-\frac{(d-1)\alpha}{d+1}}I_{\alpha-\delta}(\mu_1),
\end{align*}
where the power of $R_j$ is negative if
$d-\alpha-\frac{d-1}{2(d+1)}-\frac{(d-1)\alpha}{d+1}<0$ and $\delta>0$ is small enough, that is, $\alpha>\frac{d}{2}+\frac{1}{4}+\frac{1}{4d}.$}
\end{proof}

\section{Proof of Lemma \ref{l3}}\label{section5}\label{5}

Lemma \ref{l3} generalizes the bound on bad tubes in  \cite{du2021} from the Euclidean case to Riemannian manifolds. In Euclidean space, it follows directly from the classical Marstrand's orthogonal projection theorem and a radial projection estimate due to Orponen \cite{orponen19}. Here, we would like to pursue a slightly different approach to prove Lemma \ref{l3} on Riemannian manifolds.

Consider the origin $p_0=(0,\ldots,0)$ in our coordinate system. Let us introduce a family of submanifolds associated with each element $V$ in the Grassmannian ${\bf Gr}(n,d)$ in the tangent space of $p_0$. That is, each $V\in {\bf Gr}(n,d)$ is a $n$-dimensional linear subspace of $T_{p_0} M$. 
To prove Lemma \ref{l3}, we take $n=\lfloor d/2\rfloor$, although many arguments work for general $n$.

Fixing $V$, we consider the Fermi coordinates associated to $V$. See \cite{gray2003tubes} for a detailed discussion on Ferimi coordinates. These coordinates are of the form \[(x^V_1,\ldots,x^V_n,x^V_{n+1},\ldots,x^V_d),\] where $\exp_{p_0}(V)$ is locally given by \[\{(x^V_1,\ldots,x^V_n,0,\ldots,0):|x^V_i|<1,\,i=1,\ldots, n\}.\] Also, for each point $p=(p^V_1,\ldots,p^V_n,0,\ldots,0)\in \exp_{p_0}(V)$, the submanifold \[\Sigma_{V,p}:=\{(p^V_1,\ldots,p^V_n,x^V_{n+1},\ldots,x^V_d):|x^V_i|<1,\,i=n+1,\ldots, d\}, \] is the union of all geodesics perpendicular to $\exp_{p_0}(V)$ at the point $p$, which is locally a smooth submanifold orthogonal to $\exp_{p_0}(V)$. Moreover, $\{\Sigma_{V,p}\}_{p\in\exp_{p_0}(V)}$ is a local foliation of $M$
near the point $p_0$. Notice that for $p_1, p_2\in \exp_{p_0}(V)$ in a small neighborhood of $p_0$, 
\begin{equation}\label{equal-distance-foliation} d_g(p_1, p_2)\approx d_g(\Sigma_{V, p_1}, \Sigma_{V, p_2})\approx d_g(p, \Sigma_{V, p_2})\end{equation}
uniformly for $p\in \Sigma_{V, p_1}$. That is, $\Sigma_{V, p}$ are locally ``parallel''.

With this foliation, one can define orthogonal projection by
\[\pi_V: (x^V_1,\ldots,x^V_d)\mapsto(x^V_1,\ldots,x^V_n,0,\ldots,0).\]
Then $\pi_V\mu$ is a well defined measure on $\exp_{p_0}(V)$, and in particular
\[\pi_V\mu(p)=\int_{\Sigma_{V,p}}\mu\]
if $\mu$ has continuous density. 
In fact, we may always assume our measure has a continuous density: when dealing with $\mu(\Sigma^R)$, where $\Sigma^R$ denotes the $R^{-1}$-neighborhood of a codimension $k$ submanifold $\Sigma\subset M$, one can consider the $R^{-1}$-resolution of $\mu$, namely,
\[\mu^R(x):= \frac{\int \phi(R\cdot d_g(x,y))\,d\mu(y)}{\int \phi(R\cdot d_g(x,y))\,dy},\]
where $\phi\in C_0^\infty(\mathbb{R})$, positive. Then one can easily see that
\begin{equation}\label{localization}\mu(\Sigma^R)\lesssim R^{-k}\int_{\Sigma}\mu^{R/2} \lesssim \mu(\Sigma^{R/4}).\end{equation}
In other words when working in the scale $R^{-1}$ one can replace $\mu$ by its $R^{-1}$-resolution $\mu^R$.

Here are the key geometric features of our construction using Fermi coordinates. We claim that each geodesic segment $\gamma$ near $p_0$ is contained in some $\Sigma_{V,p}$ for suitable  $V$'s and $p$'s. Indeed, suppose that the distance between $\gamma$ and $p_0$ is minimized at the point $p_m\in\gamma$. Then for any $V$ such that $\exp_{p_0}^{-1}(p_m)\in V$ and $\exp_{p_0}(V)\perp \gamma$ at $p_m$,{ we have $\gamma\subset \Sigma_{V,p_m}$ since $\Sigma_{V,p_m}$ is exactly the union of all geodesics perpendicular to $\exp_{p_0}(V)$ at the point $p_m$.} It follows that $\pi_V(\gamma)=\{p_m\}$ is a single point and 
\begin{equation}\label{codimension-d/2-submanifold}\{V: \gamma\subset\Sigma_{V, \pi_V(\gamma)}\}=\{V: \exp_{p_0}(V)\perp\gamma
\}
\end{equation}
is a $n(d-n-1)$ dimensional submanifold of ${\bf Gr}(n,d)$. In addition, there exists a natural diffeomorphism between the Grassmannian at different points. More precisely, let $p_0'$ be another point in our coordinate system, then for every $V=V_{p_0}$ at $p_0$, since $\Sigma_{V, \pi_V (p_0')}$ is a submanifold containing $p_0'$, {the tangent space $T_{p_0'}\Sigma_{V, \pi_V (p_0')}$ is a linear subspace of $T_{p_0'}M$. Hence the map
\[\Pi_{p_0,p_0'}:V_{p_0}\mapsto \left(T_{p_0'}\Sigma_{V, \pi_V (p_0')})\right)^{\perp}=:V_{p_0'}\]
is the desired diffeomorphism. And by this definition, we have 
\begin{equation}\label{equal-submanifold}
    T_{p_0'}\Sigma_{V, \pi_V (p_0')}=V^\perp_{p'_0}
\end{equation}}
One can also define orthogonal projections as above with $p_0'$ as the origin, and in particular, \eqref{equal-distance-foliation} still holds. Recall that in our coordinate chart, $|g_{ij}-\delta_{ij}|<\epsilon_0$. We shall choose $\epsilon_0$ to be small enough such that the implicit constant in \eqref{equal-distance-foliation} is uniform under different origins. 

We remark that the foliations given by $V_{p_0}$ at $p_0$ and $V_{p_0'}$ at $p_0'$ are the same if and only if the underlying manifold is flat. 
Now we start the proof of Lemma \ref{l3}. For convenience, we denote $R:=R_j^{1/2-\delta}$.

We first decompose the manifold ${\bf Gr}(n,d)$ into a finitely overlapping union of $R^{-1}$-``balls", or more precisely, $R^{-1}$-neighborhoods of $R^{-1}$-separated $n$-dimensional subspaces. We denote the collection of the centers of these ``balls" by $\mathcal{V}_j$. Notice that $\#(\mathcal{V}_j)\approx R^{n(d-n)}$  as the dimension of ${\bf Gr}(n,d)$ is $n(d-n)$.

Next, for each $V\in\mathcal{V}_j$, one can cover the neighborhood of $p_0$ under consideration by $R^{-1}$-neighborhood of $\Sigma_{V, p_l}$, where $p_l$ are $R^{-1}$-separated points {in $\exp_{p_0}V$}. This is also a finitely overlapping cover. We denote by $\Sigma^R_{V, p}$ the $16 R^{-1}$-neighborhood of $\Sigma_{V, p}$. We say $\Sigma_{V,p}$ is bad if $\Sigma^R_{V, p}$ contains some bad tube $T\in\mathbb{T}_j$.

Now, we can estimate 
\[\mu_1\times\mu_2(\text{\rm Bad}_j)=\int\mu_2(\text{\rm Bad}_j(x))\,d\mu_1(x)=\int\sum_{x\in T\in\mathbb{T}_j,\text{bad}}\mu_2(T)\,d\mu_1(x).\]
By our discussion above, for every tube $T\in\mathbb{T}_j$ with central geodesic $\gamma$,
\[\{V: \gamma\subset\Sigma_{V, \pi_V(\gamma)}\}\subset {\bf Gr}(n,d)\]
is a $n(d-n-1)$-dimensional submanifold. Therefore, by dimension counting
\begin{equation}\label{dimension-counting}
    \#\{\Sigma_{V,p_l}: T\subset \Sigma^R_{V,p_l}\}\approx R^{n(d-n-1)}
\end{equation}
and we have
\begin{equation}\label{reduction-to-d/2-dimensional}\begin{aligned}
    \int\sum_{T:x\in T\in\mathbb{T}_j,\text{bad}}\mu_2(T)\,d\mu_1(x)=&\iint\sum_{T:x\in T\in\mathbb{T}_j,\text{bad}}\chi_T(y)\,d\mu_2(y)\,d\mu_1(x)\\ \lesssim & R^{-n(d-n-1)}\cdot\iint\sum_{l:x\in \Sigma^R_{V,p_l},\text{bad}}\chi_{\Sigma^R_{V,p_l}}(y)\,d\mu_2(y)\,d\mu_1(x)\\= & R^{-n(d-n-1)}\cdot\int\sum_{l:x\in \Sigma^R_{V,p_l},\text{bad}}\mu_2(\Sigma^R_{V,p_l})\,d\mu_1(x).\end{aligned}\end{equation}
Notice that by our definition, $\Sigma^R_{V,p_l}\subset\Sigma^{R/4}_{V,\pi_V(x)}$ whenever $\Sigma^R_{V,p_l}$ contains $x$. Therefore, the last line of \eqref{reduction-to-d/2-dimensional} is
\begin{equation}\label{reduction-to-discretie-radial-proj}\lesssim R^{-n(d-n-1)}\cdot\int\sum_{V\in\mathcal{V}_j: \Sigma_{V,\pi_V(x)} \text{  bad}}\mu_2(\Sigma^{R/4}_{V,\pi_V(x)})\,d\mu_1(x).\end{equation}

Now, we would like to change from taking sum in $V\in\mathcal{V}_j$ to taking integral in $V\in {\bf Gr}(n,d)$ over the Haar measure $\lambda$. By \eqref{localization}, for every $V'$ lying in a $R^{-1}$-neighborhood of $V$ in ${\bf Gr}(n,d)$, we have
\[\mu_2(\Sigma^{R/2}_{V,\pi_V(x)})\lesssim R^{-n}\int_{\Sigma^{R/2}_{V',\pi_{V'}(x)}}\mu_2^{R/8}=R^{-n}\cdot\pi_{V'}\mu_2^{R/8}(\pi_{V'}(x)).\]
Then we integrate in $V'\in B_{R^{-1}}(V)$ on both sides to obtain, by dimension counting,
\[\mu_2(\Sigma^{R/2}_{V,\pi_V(x)})\lesssim R^{n(d-n)-n}\cdot\int_{B_{R^{-1}}(V)}\pi_{V'}\mu_2^{R/8}(\pi_{V'}(x))\,d\lambda(V').\]
Recall elements in $\mathcal{V}_j$ are $R^{-1}$-separated, so $B_{R^{-1}}(V)$, $V\in\mathcal{V}_j$ has finite overlap. This means we have successfully reduce \eqref{reduction-to-discretie-radial-proj} to estimating
\begin{equation}\label{reduction-to-integral}\int\int_{V: \Sigma_{V,\pi_V(x)\text{ bad}}}\pi_{V}\mu_2^{R}(\pi_{V}(x))\,d\lambda(V)\,d\mu_1(x).\end{equation}

We need to understand the measure of bad $\Sigma_{V,\pi_V(x)}$, for each fixed $x$. Recall under our notation each bad $T\in \mathbb{T}_j$ satisfies {$\mu_2(4T)\geq{R_j}^{-n/2+c(\alpha)\delta}$, and the supports of $\mu_1, \mu_2$ are separated. Therefore for each $x$ the number of bad $T\in\mathbb{T}_j$ containing $x$ is $\lesssim {R_j}^{n/2-c(\alpha)\delta}\cdot R^{\delta}_j$. Also \eqref{codimension-d/2-submanifold} implies that for each $T\in \mathbb{T}_j$, the set of $V$ such that $T\subset\Sigma^{R/4}_{V,\pi_V(x)}$ lies in the $R^{-1}$-neighborhood of a codimension $n$ submanifold. Together, one can conclude that
\[\lambda\{V\in {\bf Gr}(n,d): \Sigma_{V,\pi_V(x)}\text{ is bad}\}\lesssim R^{-n}\cdot {R_j}^{n/2-c(\alpha)\delta+\delta}={R_j}^{-(c(\alpha)-(n+1))\delta}.\]}
Thus by H\"older's inequality, for every $1<q<\infty$ and $1/q+1/q'=1$, \eqref{reduction-to-integral} is bounded from above by
\begin{equation}\label{first-Holder}\begin{aligned} & {R_j}^{\frac{-(c(\alpha)-(n+1))\delta}{q'}}\cdot \int \|\pi_V\mu_2^{R}(\pi_V(x))\|_{L^q(V)}\,d\mu_1(x).\end{aligned}\end{equation}

It remains to show that there exists $q>1$ such that the integral
\begin{equation}\label{reduction-to-Lp}\int \|\pi_V\mu_2^{R}(\pi_V(x))\|_{L^q(V)}\,d\mu_1(x)<\infty,\end{equation}
and the upper bound is independent of $R$. Then the factor ${R_j}^{\frac{-(c(\alpha)-(n+1))\delta}{q'}}$ is $\lesssim {R_j}^{-2\delta}$ as desired when $c(\alpha)$ is chosen large enough. 

To prove this $L^q$ estimate, we denote
\[f(V, \pi_V(x)):=\frac{\pi_V\mu_2^{R}(\pi_V(x))^{q/q'}}{\|\pi_V\mu_2^{R}(\pi_V(x))\|_{L^q(V)}^{q/q'}}\]
and rewrite \eqref{reduction-to-Lp} into
\begin{equation}\label{reduction-to-linear}\begin{aligned}&\int \int \pi_V\mu_2^{R}(\pi_V(x))\,f(V, \pi_V(x))\,d\lambda(V)d\mu_1(x)
\\=&\int\int_V \pi_V\mu_2^{R}(p)\,f(V, p)\,d\pi_V\mu_1(p)\,d\lambda(V).\end{aligned}\end{equation}
The only property we need from the function $f$ is, for every $x$,
\[\int |f(V,\pi_V(x))|^{q'}\,d\lambda(V)=1.\]

Recall that for each $V$, under Fermi coordinates the submanifold $\exp_{p_0}(V)$ can be identified with $\mathbb{R}^{n}$. We are now ready to run some Fourier analysis. Take $n\leq\frac{d}{2}<\alpha'<\alpha$. In $\mathbb{R}^n$, it is well known that, for every compactly supported continuous function $g$ and compactly supported measure $\nu$,
\[\left(\int g\,d\nu\right)^2=\left(\int \hat{g}\cdot\overline{\hat{\nu}}\right)^2\leq \int |\widehat{g}(\xi)|^2|\xi|^{\alpha'-n}\,d\xi \cdot \int |\widehat{\nu}(\xi)|^2|\xi|^{n-\alpha'}\,d\xi.\]
By applying this inequality with $g=\pi_V\mu_2^R$ and $d\nu=f(V,{}\cdot{})\,d\pi_V\mu_1$, the square of \eqref{reduction-to-linear} is bounded from above by
\[\iint |\widehat{\pi_V\mu_2^{R}}(\xi)|^2|\xi|^{\alpha'-n}\,d\xi\,d\lambda(V)\cdot \iint |\widehat{f(V,\cdot)\,d\pi_V\mu_1}(\xi)|^2|\xi|^{n-\alpha'}\,d\xi\,d\lambda(V)=:I\cdot II\]

We consider $I$ and $II$ separately. For $I$, we shall show that the projection map $\pi_V(x)$ satisfies the transversality condition for the generalized projection estimate of Peres--Schlag \cite{R00}. More precisely, one needs to check that
$$|\nabla_V(d_g(\pi_V(x),\pi_V(y))|\gtrsim d_g(x,y)$$
at $V=V'$ if $\pi_{V'}(x)=\pi_{V'}(y)$. To see this, recall that there is a natural diffeomorphism $\Pi_{p_0,y}$ between $n$-dimensional subspaces $V_{p_0}\subset T_{p_0}M$ and $V_y\subset T_yM$. By definition of $\Pi_{p_0,y}$ we have $\pi_{V'_{p_0}}(x)=\pi_{V'_{p_0}}(y)$ { if and only if $\exp_y^{-1}(x)\perp V_y'$.
With this condition in mind, we work on the tangent space $T_yM$. One can see that, for every small $\delta>0$, there exists $V_y''$ which is $\approx \delta$ close to $V_y'$ and satisfies $|\angle(V_y'',\exp_y^{-1}(x))-\frac\pi2|\approx\delta$.  By \eqref{equal-submanifold}, the tangent space of $\Sigma_{V'',\pi_{V''}(y)}$ is exactly $(V_y'')^\perp$. Combining this fact with equation \eqref{equal-distance-foliation}, and our definition of $\pi_V$, we have
$$d_g(x,y)\cdot\delta\lesssim  d_g(\Sigma_{V'',\pi_{V''}(y)}, x)\approx d_g(\pi_{V''}(y), \pi_{V''}(x)),$$}
where $V''\subset T_{p_0}M$ is the image of $V''_y\subset T_yM$ under the diffeomorphism $\Pi^{-1}_{p_0,y}$.

Since $V_y', V_y''$ are $\delta$-close, so are $V', V''$. Hence, one can conclude that
$$|\nabla_V d_g(\pi_{V}(y), \pi_{V}(x))|\gtrsim d_g(x,y)$$
at $V'$ given $\pi_{V'}(x)=\pi_{V'}(y)$, as desired.

Once the transversality condition is satisfied, one can apply the generalized projection estimate of Peres--Schlag to obtain
\[I\lesssim I_{\alpha''}(\mu_2^R)=\iint d_g(x,y)^{-\alpha''}\mu_2^R(x)\mu_2^R(y)\,dx\,dy\]
with $n<\alpha'<\alpha''<\alpha$. One can easily check that
\begin{equation}\label{ball-condition-mu-R} \int_{B_r} \mu_2^{R}(x)\,dx\lesssim r^\alpha,\end{equation}
where the implicit constant is independent of $R$. Therefore, by considering a dyadic decomposition on $d_g(x,y)$, it follows that $I$ is finite uniformly in $R$, as desired.

It remains to estimate $II$. By standard dyadic decomposition on $|\xi|$, we obtain
\begin{align*}II&\leq \sum_{k\geq 0}2^{-(\alpha'-n)k}\int\int_{|\xi|\leq 2^k} |\widehat{f\,d\pi_V\mu_1}(\xi)|^2\,d\xi\,d\lambda(V)\\
&\lesssim \sum_{k\geq 0}2^{-(\alpha'-n)k}\int\int |\widehat{f\,d\pi_V\mu_1}(\xi)|^2\,\psi\Big(\frac{\xi}{2^k}\Big)\,d\xi\,d\lambda(V),
\end{align*}
where $\psi$ is a positive function whose Fourier transform has compact support. For each $k$, if we denote $f_V(x):=f(V,\pi_V(x))$, then
\[2^{-(\alpha'-n)k}\int\int |\widehat{f\,d\pi_V\mu_1}(\xi)|^2\,\psi\Big(\frac{\xi}{2^k}\Big)\,d\xi\,d\lambda(V)\]
\begin{align*}&=2^{-(\alpha'-n)k}\int\int \left|\int e^{-2\pi i \pi_V(x)\cdot\xi}f_V( x)\,d\mu_1(x)\right|^2\,\psi\Big(\frac{\xi}{2^k}\Big)\,d\xi\,d\lambda(V)\\
&=2^{-(\alpha'-n)k}\iiint\left(\int e^{-2\pi i (\pi_V(x)-\pi_V(y))\cdot\xi}\,\psi\Big(\frac{\xi}{2^k}\Big)\,d\xi\right)f_V(x)\,f_V(y)\,d\lambda(V)\,d\mu_1(x)\,d\mu_1(y)\\
&=2^{(2n-\alpha')k}\iint\int\hat{\psi}(2^k(\pi_V(x)-\pi_V(y)))\,f_V(x)\,f_V(y)\,d\lambda(V)\,d\mu_1(x)\,d\mu_1(y).\end{align*}
Since $\hat{\psi}$ has compact support and $\|f_V(x)f_V(y)\|_{L^{q’/2}(V)}\leq 1$ for every $x$ (we may assume $1<q<2$), by H\"older's inequality the above is

\[\leq 2^{(2n-\alpha')k}\iint\left(\lambda\{V: |\pi_V(x)-\pi_V(y)|\lesssim 2^{-k} \}\right)^{1-\frac{2}{q'}}d\mu_1(x)\,d\mu_1(y).\]

Now, we need to understand the measure of the set 
\[\{V: |\pi_V(x)-\pi_V(y)|\lesssim 2^{-k} \}.\]
Fix $x\neq y$. By the diffeomorphism $\Pi_{p_0,p_0'}$ between the Grassmannians at different points, we may take $y$ as the origin. Then by \eqref{equal-distance-foliation} we have $$|\pi_V(x)-\pi_V(y)|\approx d_g(\pi_V(x), \pi_V(y))\approx d_g(x, \Sigma_{V,\pi_V(y)})=d_g(x, \exp_y(V_y^\perp)).$$
Then, it becomes a linear algebra problem in $T_yM$. One can easily see that $|\pi_V(x)-\pi_V(y)|\lesssim 2^{-k}$ if and only if $V$ is contained in the $\min\{d_g(x,y)^{-1}2^{-k}, 1\}$-neighborhood of the codimension $n$ submanifold in \eqref{codimension-d/2-submanifold}, with $\gamma$ being the geodesic connecting $x$ and $y$. 
Therefore
$$\lambda\{V: |\pi_V(x)-\pi_V(y)|\lesssim 2^{-k} \}\lesssim \min\{d_g(x,y)^{-1}2^{-k}, 1\}^n$$
and the estimate of $II$ is reduced to
$$\sum_{k\geq 0}2^{(2n-\alpha')k}\int\int \min\{d_g(x,y)^{-1}2^{-k}, 1\}^{n(1-\frac{2}{q'})}\,d\mu_1(x)\,d\mu_1(y).$$
The next step is standard. Consider cases $d_g(x,y)\leq 2^{-k}$ and $d_g(x,y)\approx 2^{-k+j}$, $j=1,2\dots, k$ separately. By the ball condition on $\mu_1$ one can conclude that 
$$II \lesssim\sum_{k\geq 0} 2^{(2n-\alpha')k}\sum_{0\leq j\leq k}2^{-n(1-\frac{2}{q'})j+\alpha (-k+j)}=\sum_{j\geq 0}2^{(\alpha-n(1-\frac{2}{q'}))j}\sum_{k\geq j} 2^{(2n-\alpha'-\alpha)k}.$$
As $\alpha, \alpha'>n$, this is
$$\lesssim \sum_{j\geq 0} 2^{(n+\frac{2n}{q'}-\alpha')j},$$
which is finite, as desired, when $q'>0$ is large enough.

\bibliography{mybibfile}
\bibliographystyle{alpha}

\end{document}